\documentclass[12pt]{article}

\usepackage{CJK,CJKnumb,amsmath}
\usepackage{amsfonts}
\usepackage{amsthm}
\usepackage{geometry}
\usepackage{mathrsfs}
\usepackage{amsmath}
\usepackage{amssymb}
\usepackage{amsfonts}
\usepackage[percent]{overpic}
\usepackage{bm}
\usepackage[all]{xy}
\usepackage{graphicx}
\usepackage{subfigure}
\usepackage{latexsym}
\usepackage[colorlinks, linkcolor=blue, anchorcolor=blue, citecolor=blue]{hyperref}
\usepackage{epigraph}
\usepackage{hyperref}
\usepackage{fancyhdr}
\usepackage{comment}

\usepackage{mathtools}
\mathtoolsset{showonlyrefs}

\numberwithin{equation}{section}
\usepackage{color}

\theoremstyle{plain}
\newtheorem{theorem}{Theorem}[section]

\newtheorem{proposition}[theorem]{Proposition}

\theoremstyle{remark}
\newtheorem{remark}[theorem]{Remark}

\newtheorem*{ack}{Acknowledgement}

\def\sing{\operatorname{Sing}}
\def\hc{\operatorname{\widehat{\mathbb C}}}
\def\c{\operatorname{\mathbb C}}
\def\ud{\operatorname{\mathbb D}}

\def\h{\operatorname{\mathbb{H}}}

\def\f{\operatorname{\mathcal{F}}}

\def\sing{\operatorname{Sing}}

\begin{document}
\title{Siegel disks of the tangent family}
\author{Weiwei Cui and Hongming Nie}
\date{}

\maketitle

\begin{abstract}
We study Siegel disks in the dynamics of functions from the tangent family. In particular, we prove that a Siegel disk is unbounded if and only if the boundary of its image contains at least one asymptotic value. Moreover, by using quasiconformal surgery we also construct functions in the above family with bounded Siegel disks.


\end{abstract}

\section{Introduction}\label{intro}

Let $f$ be a transcendental meromorphic function in the plane. The \textit{Fatou set} of $f$ is the set of points whose iterates are defined and form a normal family in the sense of Montel. Its complement on the sphere is the \textit{Julia set}. The Fatou set is open and each component is called a Fatou component. Periodic Fatou components fall into five possible categories: attracting domains, parabolic domains, Siegel disks, Herman rings and Baker domains. The last possibility never happens for rational functions, but may occur in the transcendental setting. For more details about the Fatou and Julia sets, we refer to \cite{bergweiler93} and \cite{Milnor06}.

The dynamical behaviours of a meromorphic function $f$ are, in some sense, determined by the iterative properties of its singular values. A \textit{singular value} of $f$ is a point near which at least one branch of the inverse $f^{-1}$ is not well defined. The set of singular values of $f$, denoted by $\sing(f^{-1})$, coincides with the set of critical values and asymptotic values. Recall here that a \emph{critical value} is the image of a \textit{critical point} which has vanishing derivative and an \emph{asymptotic value} is the limit of the image of a curve tending to infinity. Note that we are including here $\infty$ as a possible singular value. For detailed connections between various Fatou components and the singular set, see \cite{bergweiler93, bergweiler95, baranski} for instance.

On a Siegel disk of period $q\ge 1$, the iteration $f^q$ is conjugate to an irrational rotation. So Siegel disks contain no critical points. On their boundaries the forward orbits of singular values are dense. However, the interplay between the singular values and the boundaries of Siegel disks is not well understood and it has attracted a lot of interest in recent years (see for example \cite{Herman85,Rogers98,petersen,Zhang11,Cheritat16,Shishikura16} for rational functions and \cite{Rempe04,Rempe08,Zhang08,bf10,zakeri10,Benini18,Cheritat18} for transcendental entire functions).

In this short paper, we study Siegel disks in the tangent family
$$\f:=\left\{f_{\lambda}(z)=\lambda\tan z: \lambda\in\c\setminus\{0\}\right\}.$$
Each map $f_\lambda$ has exactly two asymptotic values $\pm\lambda i$ and no critical values. Moreover, the map $f_\lambda$ possesses neither Herman rings, Baker domains nor wandering domains, see \cite{Devaney89}, \cite[Corollary 4]{bergweiler93} and \cite[Proposition 5.10]{Keen97}. If $f_\lambda$ has a Siegel disk, the symmetry $f_\lambda(-z)=-f_\lambda(z)$ and the orbits of the asymptotic values imply that it has no other types of periodic Fatou components,  and hence each of its Fatou components is simply connected. 
For more dynamics of $f_\lambda$, we refer to \cite{Devaney88,Devaney89,Baker92,Keen97}. We mention here that the family $\f$ is a paradigm of the class in \cite{Fagella17} consisting of transcendental meromorphic functions with finitely many singular values for which $\infty$ is not an asymptotic value.

\smallskip
\noindent{\emph{Unbounded Siegel disks.}} An irrational number is of \textit{bounded type} if the coefficients in its continued fraction expansion are bounded. A result of Graczyk and \'Swiatek \cite{Graczyk03} asserts that if a holomorphic function has a Siegel disk which is properly contained in the domain of holomorphy and has bounded type rotation number, then the Siegel disk has a critical point on the boundary. It immediately implies that every Siegel disk of $f_{\lambda}$ with rotation number of bounded type is unbounded. Indeed, note that $f_\lambda$ is holomorphic away from its poles. Suppose that $f_{\lambda}$ had a bounded Siegel disk with rotation number of bounded type. Then this Siegel disk would be compactly contained in the complement of the poles and hence its boundary would contain a critical point. This is impossible since $f_{\lambda}$ has no critical points.

For unbounded Siegel disks of the tangent family, we prove the following.

\begin{theorem}\label{thm:unbounded}
Suppose that $f_\lambda\in\f$ has a Siegel disk $\Omega$. Then $\Omega$ is unbounded if and only if $\partial f_\lambda(\Omega)\cap\sing(f_\lambda^{-1})\neq\emptyset$.
\end{theorem}

We remark here that there is no restriction on the period of the Siegel disk in our Theorem \ref{thm:unbounded}. 
To see the existence of Siegel disks of higher periods, consider the hyperbolic components (cf. ``shell components" in \cite{Fagella17}) in $\f$. Note that for any $p\ge 1$, there exist hyperbolic components, each element of which has exactly one attracting cycle of period $2p$, and hyperbolic components, each element of which has exactly two distinct symmetric attracting cycles of period $p$, see \cite[Section 8]{Keen97}. Let $\mathcal{H}$ be such a component. Then $\mathcal{H}$ is simply connected, see  \cite[Proposition 8.5]{Keen97}. Considering the multiplier map on $\mathcal{H}$, as functions in $\mathcal{H}$ move to $\partial\mathcal{H}$, we have that the attracting cycles degenerate to indifferent cycles.  Applying \cite[Theorem 6.13]{Fagella17}, we in fact obtain functions in $\partial\mathcal{H}$ possessing the aforementioned cycles whose multipliers are $e^{2\pi i\alpha}$ with Brjuno irrational numbers $\alpha$. Moreover, those cycles have the same period as the attracting cycles for functions in $\mathcal{H}$; for, otherwise, they would be parabolic. It follows immediately that such functions are linearizable at those cycles, and hence they have Siegel disks of corresponding periods.

One direction of Theorem \ref{thm:unbounded} follows easily from the fact that both asymptotic values are actually omitted: if $\partial f_\lambda(\Omega)\cap\sing(f_\lambda^{-1})\neq\emptyset$ and $\Omega$ is bounded, then there exists a finite point in $\partial\Omega$ whose image is an asymptotic value, which is impossible. The  reverse implication is similar to Rempe's result \cite{Rempe04} where he proved that an unbounded Siegel disk in the exponential family contains the finite asymptotic value on the boundary of its orbit.  To prove the above result, we suppose by contrary that $\partial f_\lambda(\Omega)\cap\sing(f_\lambda^{-1})=\emptyset$. Then consider two suitable disjoint closed disks around the two asymptotic values. The preimage of the complement of these disks is a horizontal strip containing the real axis. By constructing simple curves connecting  the boundaries of the two disks and considering their preimages, we can reach a contradiction due to the unboundedness of $\Omega$.

\medskip
A transcendental meromorphic function $f$ with two singular values is of the form $M\circ\exp\circ A$, where $M$ is M\"obius and $A$ is linear (equivalently, one can also say that $f$ is of the form $M\circ\tan\circ A$), see Proposition \ref{form}. Combining our method here with Rempe's method in \cite{Rempe04}, one can actually obtain the following theorem. Indeed, if $f$ has  two singular values, at least one of them is finite. Our method deals with the case that both singular values are finite, while Rempe's method works in the remaining case.

\begin{theorem}\label{Thm2}
Let $f:\c\to\hc$ be a transcendental meromorphic function with exactly two singular values. Suppose that $f$ has a Siegel disk $\Omega$ of period $q\ge 1$. Then $\bigcup_{j=0}^{q-1}f^j(\Omega)$ is unbounded if and only if $\left(\bigcup_{j=0}^{q-1}\partial f^j(\Omega)\right)\cap\sing(f^{-1})\neq\emptyset$.
\end{theorem}

\medskip
\noindent{\emph{Bounded Siegel disks.}} For the existence of bounded Siegel disks, we construct a map in $\mathcal{F}$ having a bounded Siegel disk with quasicircle boundary. Using quasiconformal surgery, we show

\begin{theorem}\label{thm:bounded}
There exists $\lambda\in\mathbb{S}^1$ such that $f_{\lambda}\in\f$ has a bounded Siegel disk $\Omega$ around $0$ with quasicircle boundary $\partial\Omega$ and $\partial\Omega\cap\sing(f_\lambda^{-1})=\emptyset$.
\end{theorem}

\section{Unbounded Siegel disks}

In this section, we prove %\hmn{Theorems \ref{thm:unbounded} and \ref{Thm2}}. \cww{theorem 1.2 is not proved later. So probably only say that we prove Theorem 1.} {\tiny 
that the boundary of the image of an unbounded Siegel disk for $f_{\lambda}(z)=\lambda\tan z$ must intersect with the singular set. We will use the notations $D(a,r)$ (resp. $\overline{D}(a,r)$) for the open (resp. closed) euclidean disk of radius $r$ around $a$. Let $\ud_r:=D(0,r)$ and write $\ud=\ud_1$ for simplicity. For $\delta\in\mathbb{R}$, put $\h^\delta:=\{z=x+iy: y>\delta\}$ and  $\h_\delta:=\{z=x+iy: y<\delta\}$. Moreover, for $R>0$, denote $S_R:=\mathbb{C}\setminus\overline{\h^R\cup\h_{-R}}$ the horizontal strip of width $2R$ symmetric with respect to the real axis.

First we prove the following basic result for $f_{\lambda}$. %For this purpose, we put, for some $r>0$,
%$$D_{1}:=D(i\lambda, r)~\,~\text{and}\,~D_{2}:=D(-i\lambda, r).$$

\begin{proposition}\label{2.1}
Let $W$ be a domain in $\c$. If $\overline{f_{\lambda}(W)}\,\cap\,\sing(f_{\lambda}^{-1})=\emptyset$, then $W$ is contained in $S_R$ for some $R>0$.
\end{proposition}

\begin{proof}
By assumption, for some $r>0$ we can choose two disks $D_1:=D(i\lambda, r)$ and $D_2:=D(-i\lambda, r)$ with disjoint closures such that $D_k \cap \overline{f_{\lambda}(W)}=\emptyset$ for $k=1,2$. It is clear that for each $k$, the domain $f_{\lambda}^{-1}(D_k)$ is an upper or lower half plane. This can be seen by considering $f_{\lambda}=M_1\circ\exp\circ M_2$, where $M_1(z)=-\lambda i(z-1)/(z+1)$ and $M_2(z)=2iz$. Therefore, there exists $R>0$ (depending only on $r$) such that $f_{\lambda}\left(\overline{\h^R\cup \h_{-R}}\right)=\overline{D_1\cup D_2}$. In particular, this implies that $W\subset S_R$.
\end{proof}

Now we prove Theorem \ref{thm:unbounded}.

\begin{proof}[Proof of Theorem \ref{thm:unbounded}]
As mentioned in Section \ref{intro}, it suffices to show that if $\Omega$ is unbounded, then $\partial f_\lambda(\Omega)\cap\sing(f_\lambda^{-1})\neq\emptyset$. Suppose to the contrary that $\partial f_\lambda(\Omega)\cap\sing(f_\lambda^{-1})=\emptyset$. Note that $\sing(f_{\lambda}^{-1})=\{i\lambda,-i\lambda\}$. Fix $r>0$ sufficiently small and set
$$D_1:=D(i\lambda, r),\,~\, ~D_2:=D(-i\lambda, r)$$
such that $\overline{D}_1\cap\overline{D}_2=\emptyset$ and $f_\lambda\left(\Omega\right)\cap(\overline{D}_1\cup\overline{D}_2)=\emptyset$.
Define
$$V:=\hc\setminus(\overline{D_1\cup D_2})~\,~\text{and}~\,~U:=f_{\lambda}^{-1}(V).$$
Since $V\cap\sing(f_{\lambda}^{-1})=\emptyset$ , the map $f_\lambda: U\to V$ is a covering map. From the proof of Proposition \ref{2.1}, there exists $R>0$ such that $U=S_R$.

Let $\phi:\mathbb{D}\to\Omega$ and $\psi:\mathbb{D}\to f_\lambda(\Omega)$ be Riemann maps such that $\phi(0)=o$ and $\psi(0)=f_\lambda(o)$,
where $o$ is the center of $\Omega$. It follows that $\psi^{-1}\circ f_\lambda\circ\phi(z)$ is an automorphism of $\mathbb{D}$ fixing $0$. Hence there exists $\rho\in\mathbb{C}$ with $|\rho|=1$ such that
 $$\psi^{-1}\circ f_\lambda\circ\phi(z)=\rho z.$$
 Pick $0<r_1<1$ and consider $W^{(1)}:=\psi(\mathbb{D}_{r_1})$. It follows that  $W^{(1)}\subset f_\lambda(\Omega)$ and  $\partial W^{(1)}$ is a Jordan curve. Let $a_1\in \partial D_1$ and $a_2\in\partial D_2$ be two points. Then there exists a simple curve $\Gamma_1\subset\mathbb{C}\setminus(D_1\cup D_2)$ connecting $a_1$ and $a_2$ such that $\Gamma_1\cap W^{(1)}\not=\emptyset$ and $\Gamma_1\cap\partial W^{(1)}$ contains exactly two points. Indeed, to obtain such a $\Gamma_1$, we take the image of a diameter line in $\overline{\mathbb{D}}_{r_1}$ under $\psi$, and then continue it in $\mathbb{C}\setminus(D_1\cup D_2\cup\overline{W^{(1)}})$ until reaching $a_1$ and $a_2$. See Figure \ref{cc}.
 
Now consider the preimage $f_\lambda^{-1}(\Gamma_1)$. There exists a bounded simple curve $\alpha_0$ connecting both boundaries of $S_R$ such that $f_\lambda^{-1}(\Gamma_1)=\cup_{j\in\mathbb{Z}}\alpha_j$  with $\alpha_j=\alpha_0+j\pi$. It follows that $\alpha_j\cap\alpha_\ell=\emptyset$ for $j\not=\ell$, and hence $f_\lambda$ is injective on $\alpha_j$.
%by $\alpha_j$ with $j\in\mathbb{Z}$ such that $\alpha_j=\alpha_0+j\pi$. Note that each $\alpha_j$ is a bounded simple curve connecting both boundaries of $S_R$. Moreover, $\alpha_j\not=\alpha_\ell$ for $j\not=\ell$. 
Set  $\Omega^{(1)}:=\phi(\mathbb{D}_{r_1})$. Then $W^{(1)}=f_\lambda(\Omega^{(1)})$ and $\Omega^{(1)}\cap f_\lambda^{-1}(\Gamma_1)\not=\emptyset$. Without loss of generality, we can assume $\Omega^{(1)}\cap\alpha_0\not=\emptyset$. By the choice of $\Gamma_1$, we have $f_\lambda(\Omega^{(1)}\cap\alpha_0)=W^{(1)}\cap\Gamma_1$. Then $\Omega^{(1)}\cap\alpha_j=\emptyset$ for all $j\in\mathbb{Z}\setminus\{0\}$ since $f_\lambda$ is injective on $\Omega$.

Since $\Omega$ is unbounded, there exists $r_1<r_2<1$ such that
$$\left(\Omega^{(2)}\cap\alpha_2\right)\bigcup\left(\Omega^{(2)}\cap\alpha_{-2}\right)\not=\emptyset,$$
where $\Omega^{(2)}:=\phi(\mathbb{D}_{r_2})$. Without loss of generality, we can assume $\Omega^{(2)}\cap\alpha_2\not=\emptyset$. It follows that $\Omega^{(2)}\cap\alpha_1\not=\emptyset$. Set $W^{(2)}:=\psi(\mathbb{D}_{r_2})$. Note that $W^{(1)}\subset W^{(2)}\subset f_\lambda(\Omega)$ and  $\partial W^{(2)}$ is a Jordan curve. We can choose a simple curve $\Gamma_2\subset\mathbb{C}\setminus(D_1\cup D_2)$ connecting $a_1$ and $a_2$ such that $\Gamma_2\cap\partial W^{(2)}$ contains exactly two points and 
$$\Gamma_1\cap\overline{W^{(1)}}=\Gamma_2\cap\overline{W^{(1)}}=\left(\Gamma_1\cap\Gamma_2\right)\setminus\{a_1,a_2\}.$$ 
Indeed, to obtain such a $\Gamma_2$, we continue $\Gamma_1\cap\overline{W^{(1)}}$ in $\overline{W^{(2)}}\setminus\overline{W^{(1)}}$ until reaching points in $\partial W^{(2)}$, and then continue it in $\mathbb{C}\setminus(D_1\cup D_2\cup \overline{W^{(2)}})$ until reaching $a_1$ and $a_2$.

For the preimage $f_\lambda^{-1}(\Gamma_2)$, there exists a bounded simple curve $\beta_0$ connecting both boundaries of $S_R$ such that $\alpha_0\cap\Omega^{(1)}=\beta_0\cap\Omega^{(1)}$ and $f_\lambda^{-1}(\Gamma_2)=\cup_{j\in\mathbb{Z}}\beta_j$ with $\beta_j=\beta_0+j\pi$. Then $\alpha_j\cap\beta_j\not=\emptyset$ and $\beta_j\cap\beta_\ell=\emptyset$ for $j\not=\ell$. Moreover,  $f_\lambda$ is injective on $\beta_j$. 
%Note that $f_\lambda(\alpha_j\cap\beta_j)=f_\lambda(\alpha_0\cap\beta_0)=\Gamma_1\cap\overline{W^{(1)}}$. It follows that $f_\lambda(\beta_j\setminus\alpha_j)\subset\Gamma_2\setminus\overline{W^{(1)}}$.
%Denote the preimage $f_\lambda^{-1}(\Gamma_2)$ by $\beta_j$ with $j\in\mathbb{Z}$ such that $\beta_j=\beta_0+j\pi$ and $\beta_j\cap\alpha_j\not=\emptyset$. Then each $\beta_j$ is a bounded simple curve connecting both boundaries of $S_R$ and $\beta_j\not=\beta_\ell$ for $j\not=\ell$. Now 
%In fact, we can choose $\Gamma_1$ and $\Gamma_2$ such that they are sufficiently close. Indeed, let $W$ be a small neighborhood of $W^{(1)}$ such that $W^{(1)}\subset W\subset W^{(2)}$. We can 
%Then $\alpha_0\cap f_\lambda^{-1}(\Gamma_2)\subset\beta_0$ and hence $\alpha_j\cap f_\lambda^{-1}(\Gamma_2)\subset\beta_j$ for all $j\in\mathbb{Z}$. 

We claim that $\beta_j\cap\alpha_\ell\cap S_R=\emptyset$ if $j\not=\ell$. It suffices to show the claim holds for $\ell=0$ since $\alpha_\ell=\alpha_0+\ell\pi$. Suppose that there exists $j_0\in\mathbb{Z}\setminus\{0\}$ such that $\beta_{j_0}\cap\alpha_0\cap S_R\not=\emptyset$. It follows that $(\beta_{j_0}\setminus\alpha_{j_0})\cap\alpha_0\cap S_R\not=\emptyset$ since $\alpha_{j_0}\cap\alpha_0=\emptyset$. Pick $z\in(\beta_{j_0}\setminus\alpha_{j_0})\cap\alpha_0\cap S_R$ and note that for all $j\in\mathbb{Z}$,
$$f_\lambda\left(\alpha_j\cap\beta_j\cap S_R\right)=f_\lambda\left(\alpha_0\cap\beta_0\cap S_R\right)=\Gamma_1\cap\overline{W^{(1)}}.$$ 
We have that $f_\lambda(z)\in(\Gamma_2\setminus\Gamma_1)\cap\Gamma_1$. It is impossible. Thus the claim holds. 

To finish the proof, we can derive a contradiction by locating $\beta_1$. %Now we consider the curve $\beta_1$. 
The above claim implies that $\beta_1\cap S_R$ is contained in the interior of the quadrilateral bounded by the $\alpha_0,\alpha_2$ and the two boundaries of $S_R$. The choice of $r_2$ and the properties of $\beta_1$ guarantee that $\beta_1\cap\Omega^{(2)}\not=\emptyset$. However, note that $f_\lambda(\beta_0\cap\Omega^{(2)})=\Gamma_2\cap W^{(2)}$ by the choice of $\Gamma_2$. Since $f_\lambda$ is injective on $\Omega$, we have $\beta_1\cap\Omega^{(2)}=\emptyset$. It is a contradiction.
\end{proof}

\begin{figure}[htbp] %  figure placement: here, top, bottom, or page
    \centering
     \includegraphics[width=14.5cm]{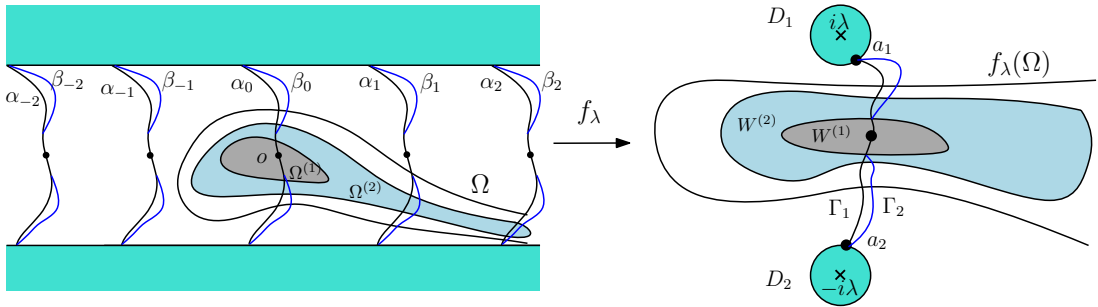}
     \caption{Construction of $\Gamma_1$ and $\Gamma_2$ and their preimages. As shown, $\Gamma_1$ (black solid curve) is a simple curve connecting $a_1$ and $a_2$ while $\Gamma_2$ consists of the part of $\Gamma_1$ in $W^{(1)}$ and two blue simple curves.}
     \label{cc}
\end{figure}

\begin{remark}
In principle, the only requirements of $\Omega$ in the proof of Theorem \ref{thm:unbounded} are that $\Omega$ is simply connected and $f$ is univalent in $\Omega$. In particular, if  $\Omega$ is a simply connected Fatou component of $f_\lambda$ in which $f$ is univalent, Theorem \ref{thm:unbounded} holds. %\cww{We also need the unboundedness of $\Omega$. Siegel disks are foliated by invariant curves, which are also important for us. If you only have a Fatou component, maybe we cannot achieve the same result.}
\end{remark}

%\hmn{As mentioned in Section \ref{intro}, Theorem \ref{Thm2} follows from the combination of the method in \cite{Rempe04}, our method in the above proof of Theorem \ref{thm:unbounded} and the following result, which seems to be classical but is not easy to locate a reference.
%{\tiny For the transcendental meromorphic functions with exactly two singular values, an explicit form for such functions is given in the following result.}  \cww{We are grateful to Walter Bergweiler for suggesting an idea of proof. For a slightly different proof, see \cite[Theorem 2.1]{ac20}.}}
%The following result seems to be classical but we are unable to locate a reference. We will outline a short proof of this fact which was suggested by Walter Bergweiler.

Our next result concerns the form of transcendental meromorphic functions with exactly two singular values, which, as mentioned in Section \ref{intro}, is used to prove Theorem \ref{Thm2}. It seems to be classical but is not easy to locate a reference.  The idea of our proof is suggested by Walter Bergweiler. For a slightly different proof, see \cite[Theorem 2.1]{ac20}.
\begin{proposition}\label{form}
Let $f:\c\to\hc$ be transcendental and meromorphic with exactly two singular values. Then $f$ is of the form $M\circ\exp\circ A$, where $M$ is M\"obius and $A$ is linear.
\end{proposition}
\begin{proof}
Post-composing a M\"obius transformation, we can assume that the singular values of $f$ are at $0$ and $\infty$. Then $f^{-1}(\exp)$ can be continued analytically in the whole plane and thus equals to some entire function $h$. It follows that $\exp=f\circ h$. Thus, to show the existence of $A$, it suffices to prove that $h$ is linear.

Note that $\mathrm{Sing}(\exp^{-1})=\{0,\infty\}$, and both $0$ and $\infty$ are asymptotic values of $\exp$. Then $f$ has at least two asymptotic values at $0$ and $\infty$. Since $f$ has exactly two singular values, it follows that $\mathrm{Sing}(f^{-1})=\{0,\infty\}$. Note that $h$ is non-constant. By the little Picard theorem, the map $h$ omits at most one point in $\mathbb{C}$. If $h$ has no omitted values (i.e., $h$ is onto), then $f$ will omit $0$ and $\infty$ as $\exp$ does. If $h$ has exactly one omitted value, say $w_0$. Then $f(w_0)$ is equal to either $0$ or $\infty$, since otherwise $f\circ h$, i.e., $\exp$, will have a third singular value. 
 By considering $1/f$ if necessary, %Suppose without loss of generality 
 we may assume that $f(w_0)=0$. Then $\infty$ will be omitted by $f$. This is because any pole of $f$ will otherwise have a preimage under $h$, which is impossible since $\exp$ omits $\infty$.
%{\tiny If $h(\mathbb{C})=\mathbb{C}$, since $\mathrm{exp}$ omits both $0$ and $\infty$, so does $f$.  If $h$ omits exactly one point $z_0$ in $\mathbb{C}$, again since $\mathrm{exp}$ omits both $0$ and $\infty$, then $f$ omits either $0$, $\infty$ or both depending on that $f(h(z_0))$ equals to $\infty$, $0$ or some non-zero point in $\mathbb{C}$. }
%Thus by considering $1/f$ if necessary, we may assume that $\infty$ is an omitted point of $f$. Then 
It follows that $f$ is an entire function. By the Denjoy-Carleman-Ahlfors theorem, the order of $f$ is at least $1/2$. So $h$ is a polynomial; for, otherwise, a result of P\'olya \cite{polya}, which states that for two entire functions $f_1$ and $f_2$, the composition $f_1\circ f_2$ has infinity order unless $f_1$ is of finite order and $f_2$ is a polynomial or $f_1$ has order zero and $f_2$ is of finite order, implies that $f\circ h$ will have infinity order,  while $\exp$ has order $1$. Since $\exp$ is locally univalent in the whole plane, we in fact have that $h$ is linear. This completes the proof.
\end{proof}

%{\color{blue}I have a short proof: since $f$ has two asymptotic values, then the order of $f$ is at least $1$ by the Denjoy-Carleman-Ahlfors theorem. So $h$ is a polynomial, since otherwise we can apply a result of Edrei and Fuchs to conclude that $f\circ h$ has infinity order. A contradiction. That $h$ is linear follows from the local univalence of $\exp$ in the plane.}

\begin{remark}%{\color{blue}maybe we do not need this?}
Proposition \ref{form} implies that %for any transcendental  meromorphic function with exactly two singular values, then both singular values are asymptotic values. In particular, 
there does not exist a meromorphic function with exactly one critical value and one asymptotic value. %, counted without multiplicity.
\end{remark}

\section{Bounded Siegel disks}
In this section, we prove Theorem \ref{thm:bounded}. It is a standard application of quasiconformal surgery. For similar applications, we refer to \cite{Rempe03} for the exponential family and \cite{Zhang08} for the
sine family.

\begin{proof}[Proof of Theorem \ref{thm:bounded}]
Let $\theta_0$ be an irrational number of bounded type and set $\lambda_0=e^{2\pi i\theta_0}$. Then $f_{\lambda_0}$ has an (unbounded) Siegel disk $\Omega_0$ around $0$. On $\Omega_0$, the map $f_{\lambda_0}$ is conjugate to an irrational rotation by a biholomorphic map $\varphi:\mathbb{D}\to\Omega_0$. Since $f_{\lambda_0}(-z)=-f_{\lambda_0}(z)$, the Siegel disk $\Omega_0$ is symmetric about $0$ and hence $\varphi$ is odd.  Fix some $0<r<1$ and set $K=\varphi(\overline{\mathbb{D}}_r)$. It follows that $K$ is symmetric about $0$.

%{\tiny We claim that the region $K$ is symmetric about $0$. Let $-K$ be the symmetric region of $K$ about $0$. Note that $0\in K\cap(-K)$. Then there exists $z\in\partial K\cap\partial(-K)$. It follows that $-z\in \partial K$ since $z\in \partial(-K)$, and $-z\in \partial(-K)$ since $z\in \partial K$. Thus $-z\in \partial K\cap\partial(-K)$. Since on $\partial K$ the map $f_{\lambda_0}$ is conjugate to the above irrational rotation, the closure of the forward orbit of $z$ is $\partial K$. It follows that the closure of the forward orbit of $-z$ is $\partial(-K)$. Note that the closure of the forward orbit of $-z$ is $\partial K$ since $-z\in\partial K$. It implies that $\partial K=\partial(-K)$ and hence $K=-K$.}

Consider the Riemann map $\Phi:\widehat{\mathbb{C}}\setminus\overline{\mathbb{D}}\to\widehat{\mathbb{C}}\setminus K$ fixing $\infty$. Note that $\partial K=\varphi(\partial\mathbb{D}_r)$. Then $\Phi$ extends to a $C^\infty$ map on $\widehat{\mathbb{C}}\setminus\mathbb{D}$, see \cite[Theorem A]{Bell87}. Moreover, the map $\Phi$ is odd since $K$ is symmetric about $0$. Furthermore, the map $\Phi$ extends to an odd quasiconformal map on $\widehat{\mathbb{C}}$,  see \cite[Proposition 2.30]{Branner14}, which we still call the extension $\Phi$. On $\mathbb{C}$, define
$$G(z):=\Phi^{-1}\circ f_{\lambda_0}\circ\Phi(z).$$
It follows that $G|_{\mathbb{S}^1}$ is a $C^\infty$ circle diffeomorphism. By a result of Herman \cite{Herman86}, also see \cite[Theorem 3.21]{Branner14}, there exists $\lambda_1\in\mathbb{S}^1$ such that $(\lambda_1G)|_{\mathbb{S}^1}$ is quasisymmetrically conjugate to an irrational rigid rotation $R_\theta$, but not $C^2$ conjugate. Denote this quasisymmetric conjugacy by $h$ and post-compose with a rotation such that $h(1)=1$. Following from the proof of \cite[Lemma 2.6]{Zhang08}, we have that the map $h$ is odd. Let $H:\overline{\mathbb{D}}\to\overline{\mathbb{D}}$ be the Douady-Earle extension of $h$, see \cite{Douady86}. Then $H$ is also odd.
Define
$$g(z):=
\begin{cases}
\lambda_1\,G(z), &\text{~if}\  z\in\mathbb{C}\setminus\mathbb{D},\\
H^{-1}\circ R_\theta\circ H, &\text{~if}\  z\in\mathbb{D}.
\end{cases}
$$

Now we pull back the standard complex structure by $H$ and obtain a complex structure in $\mathbb{D}$ denoted by $\nu_0$. Let $\nu$ be the complex structure on $\mathbb{C}$ defined as follows. For $z\in\mathbb{C}$, if the forward $g$-orbit of $z$ intersects with $\mathbb{D}$, then $\nu(z)$ is the pull back of $\nu_0$ along the orbit. Otherwise, put $\nu(z)=0$. Then $\nu$ is $g$-invariant and $\nu(-z)=\nu(z)$. We claim that $||\nu|
|_\infty$ is bounded. Indeed, note that $g(\mathbb{S}^1)=\mathbb{S}^1$. Then any point $z\in\mathbb{C}$ whose $g$-orbit intersects with $\mathbb{D}$ is either in $\mathbb{D}$ or in $\mathbb{C}\setminus\overline{\mathbb{D}}$. From the construction and the fact that $g$ is holomorphic on $\mathbb{C}\setminus\overline{\mathbb{D}}$, we know that $|\nu|$ only depends on $H$ and hence has a uniform upper bound on $\mathbb{C}$. Thus the claim holds.

Let $\psi$ be a quasiconformal homeomorphism of the Riemann sphere which solves the Beltrami equation given by $\nu$ and fixes $0$ and $\infty$, see \cite{Ahlfors06}. It follows that $\psi$ is odd, see \cite[Lemma 2.10]{Zhang08}. Now set $f=\psi\circ g\circ\psi^{-1}$. Then $f$ is a transcendental meromorphic function having a Siegel disk $\Omega_1$ centred at $0$ and $\psi(\mathbb{D})\subset\Omega_1$. We claim that $\psi(\mathbb{D})=\Omega_1$, and hence $\partial\Omega_1=\partial\psi(\mathbb{D})$ is a quasicircle and contains no asymptotic values. Suppose $\psi(\mathbb{D})\subsetneq\Omega_1$. Then the map $H\circ\psi^{-1}$ extends to a conformal map on $\Omega_1$. Since $\psi(\mathbb{S}^1)\subset\Omega_1$, it follows that $(\psi\circ\lambda_1 G\circ\psi^{-1})|_{\psi(\mathbb{S}^1)}$ is analytically conjugate to $R_\theta$. By the definition of $\psi$, we know that $\psi$ is analytic on $\mathbb{S}^1$. Hence $\lambda_1G|_{\mathbb{S}^1}$ is analytically conjugate to $R_\theta$. This contradicts with the choice of $\lambda_1$.

Now we claim that there exists $a,b\in\mathbb{C}\setminus\{0\}$ such that $f(z)=b\tan{az}$. Note that from the construction, the map $f$ is a meromorphic function with exactly two singular values such that $f(-z)=-f(z)$ and $f(0)=0$. Then there exist a M\"obius transformation $M$ and a linear map $A$ such that $f(z)=M\circ \tan \circ A(z)$. Note that $e^{\alpha z+\beta} =e^\beta e^{\alpha z}$ for $\alpha\in\mathbb{C}\setminus\{0\}$ and $\beta\in\mathbb{C}$. We can assume $A(z)=az$ for some $a\in\mathbb{C}\setminus\{0\}$. Since $f(-z)=-f(z)$, we have $M(z)$ is odd. It follows that either $M(z)=bz$ or $M(z)=b/z$ for some $b\in\mathbb{C}\setminus\{0\}$.  Since $f(0)=0$, we have in fact $M(z)=bz$. The claim holds. 

Since $a\not=0$, set $L(z)=z/a$ and $F(z)=L^{-1}\circ f\circ L(z)$. Then $F(z)=ab\tan z$. Note that $L(0)=0$. The map $F$ has a Siegel disk $\Omega:=L^{-1}\circ\psi(\mathbb{D})$ around $0$. Moreover, the multiplier of $F$ at the fixed point $0$ is $ab$. It follows that $ab\in\mathbb{S}^1$. Set $\lambda:=ab$. Then $f_\lambda:=F$ is a desired map.
\end{proof}

\begin{ack}
The authors thank Weixiao Shen, Gaofei Zhang and Jianhua Zheng for useful discussion and helpful comments. The authors also thank the referees for careful reading and helpful suggestions. The first author was partially supported by the China Postdoctoral Science Foundation (No. 2019M651329). The second author was partially supported by ISF Grant 1226/17.
\end{ack}

\bibliographystyle{siam}
%\bibliography{references}

\bigskip
\noindent Weiwei Cui:

\emph{Shanghai Center for Mathematical Sciences,  Fudan University,  2005 Songhu Road, Shanghai 200438,  China}; and

\emph{Centre for Mathematical Sciences, Lund University, Box 118, 22 100 Lund, Sweden}

\medskip
\emph{cuiweiwei@fudan.edu.cn}

\bigskip
\noindent Hongming Nie:

\emph{ Einstein Institute of Mathematics, The Hebrew University of Jerusalem, Givat Ram. Jerusalem, 9190401, Israel}

\medskip
\emph{hongming.nie@mail.huji.ac.il}

\end{document}